\documentclass[12pt]{article}

\usepackage{authblk}

\usepackage{amsfonts}
\usepackage{graphicx}
\usepackage{tabularx}
\usepackage{array}
\usepackage[usenames,dvipsnames]{color}
\usepackage{comment}
\usepackage{amsmath}
\usepackage{amsthm}
\usepackage{amssymb}
\usepackage{fullpage}
\usepackage[dvipsnames]{xcolor}
\usepackage{tikz}
\usepackage{listings}
\usepackage{dsfont}
\usepackage{enumerate}
\usepackage{romannum}

\newtheorem{theorem}{Theorem}[section]

\newtheorem{conjecture}[theorem]{Conjecture}
\newtheorem{problem}[theorem]{Problem}

\newtheorem{lemma}[theorem]{Lemma}

\theoremstyle{definition}

\def\epsilon{\varepsilon}
\def\claw{\text{claw}}
\def\butterfly{\text{butterfly}}
\def\E{\textit{E}}
\DeclareMathOperator{\dist}{dist}

\title{Cops and Robbers on Graphs with Path Constraints}
\author[1]{Alexander Clow\,\thanks{Supported by the Natural Sciences and Engineering Research Council of Canada (NSERC) through PGS D-601066-2025}}
\author[2]{Erin Meger\, \thanks{Supported by  the Natural Sciences and Engineering Research Council of Canada NSERC (2025-05700) and by Queen’s University.}}
\date{}

\affil[1]{ \small{Department of Mathematics, Simon Fraser University}}
\affil[2]{ \small{School of Computing, Queens University}}

\begin{document}

\maketitle

\begin{abstract}
    In 2019, Sivaraman conjectured that every $P_k$-free graph has cop number at most $k-3$.
    In the same year, Liu proved this conjecture for $(P_k,\claw)$-free graphs.
    Recently Chudnovsky, Norin, Seymour, and Turcotte proved this conjecture for $P_5$-free graphs.
    For $k\geq 6$ the conjecture remains widely opened.
    Let the $\E$ graph be the $\claw$ with two subdivided edges.
    We show that all $(P_k,\E \,)$-free graphs have cop number at most $\lceil \frac{k-1}{2} \rceil +3$, 
    which improves and generalizes Liu's result for $(P_k,\claw)$-free graphs.
    We also prove that if $G$ is a graph whose longest path is length $p$, then 
    $G$ has cop number at most $\lceil \frac{2p}{3} \rceil+3$.
    This improves a bound of Joret, Kami\'nski, and Theis.
    Our proof relies on demonstrating that all $(P_k,\claw,\butterfly,C_4,C_5)$-free graphs have cop number at most $\lceil\frac{k-1}{3}\rceil +3$.
\end{abstract}

\section{Introduction}
\pagenumbering{arabic}

Cops and Robbers is a two-player game played on a connected graph, see \cite{AIGNER1984, nowakowski1983vertex,quilliot1983problemes}. 
To begin the game, the cop player places $k$ cops onto vertices of the graph,
then the robber player chooses a vertex to place the robber. 
Players take turns moving. 
During the cop player’s turn, each cop either moves to an adjacent vertex or passes and remains at their current vertex. 
Similarly, on the robber player’s turn, the robber either moves to an adjacent vertex or passes and remains at their current vertex.
The cop player wins if after finitely many moves a cop can move onto the vertex occupied by the robber, called capturing.
The robber player wins if the robber can provide a strategy to evade capture indefinitely. 
The least number of cops required for the cop player to win, regardless of the robber’s strategy, is called  the cop number of a graph, denoted $c(G)$ for a graph $G$.
If $c(G)\leq k$, then we say $G$ is $k$-cop win.
We suppose all graphs are connected, unless stated otherwise.
For more background on Cops and Robbers we recommend \cite{bonato2011game}.

We define the graphs $P_t$, $C_t$, and $K_t$ as the path with $t$ vertices, the cycle with $t$ vertices, and the complete graph on $t$ vertices, respectively.
We call the complete bipartite graph $K_{1,3}$ the \textit{claw}
, and we call the $5$ vertex graph given by identifying two triangles at a vertex the \textit{butterfly}.
The $\E$ graph is give by subdividing two edges of the $\claw$.
If $G$ and $H$ are graphs, then we denote the disjoint union of $G$ and $H$ by $G+H$.
For any graph $G$, we use $mG$ to denote the disjoint union of $m$ copies of $G$, and we let $\overline{G}$ denote the compliment of $G$.
We let $\alpha(G)$ denote the independence number of $G$.
For more background and definitions in graph theory we refer the reader to \cite{west2001introduction}.

\begin{figure}[h!]
    \centering
    \includegraphics[scale = 1.0]{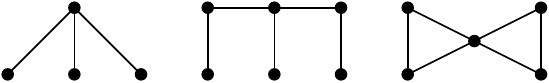}
    \caption{From left to right the $\claw$, $\E$, and $\butterfly$ graphs are shown.}
    \label{fig:small graphs}
\end{figure}

It is standard to consider classes of graphs defined by forbidden substructures such as minors or induced subgraphs.
A graph $G$ is $H$-free or $H$-minor free if $G$ does not
contain, respectively, any induced subgraph or minor which is isomorphic to $H$. 
Of particular importance for this paper will be the class of $P_k$-free graphs.
Also of note are graph with independence number at most $k$, 
which is equivalently the class of $(k+1)K_1$-free graphs.
If $\mathcal{H}$ is a list of graphs, we say $G$ is $\mathcal{H}$-free if $G$ has no graph in $\mathcal{H}$ as an induced subgraph.
Also observe $G$ having a longest path of length $p < \ell$ is equivalent to having no $P_\ell$ as a subgraph.

Cops and Robbers has been studied extensively in relation to forbidden minors and forbidden induced subgraphs.
Andreae \cite{andreae1986pursuit} showed that for all graphs $H$, there exists a constant $m = m(H)$
such that if $G$ has no $H$-minor, then $c(G) \leq m$.
The constant $m = m(H)$ here was recently improved by Kenter, Meger, and Turcotte \cite{kenter2025improved},
particularly for small or sparse graphs $H$.
When forbidding an induced subgraph $H$, or a subgraph $H$, one cannot guarantee the cop number of $H$-free graphs is bounded.
In fact, it was shown by Joret, Kami{\'n}ski, and Theis \cite{joret2010cops} 
that $H$-free graphs have bounded cop number if and only if $H$ is a linear forest (i.e a disjoint union of paths).
This characterization of which forbidden induced subgraphs admit classes with bounded cop number was extended by Masjoody and Stacho \cite{masjoody2020cops}
who demonstrated analogous results when forbidding multiple induced subgraphs simultaneously.

In \cite{joret2010cops} Joret, Kami{\'n}ski, and Theis demonstrated that 
every connected $P_k$-free graph is $(k-2)$-cop win.
However, their argument missed key details.
This lead Sivaraman \cite{sivaraman2019application} to give a different and simpler proof that every connected $P_k$-free graph is $(k-2)$-cop win.
This simpler proof involves the well known Gy\'arf\'as path argument,
and has the added benefit of bounding the number of turns it takes the cops to catch the robber.

When $k$ is large there is no evidence that this bound is optimal.
This lead Sivaraman to the following conjecture.

\begin{conjecture}[Sivaraman \cite{sivaraman2019application}]\label{Conj: main}
    For all $k\geq 5$, if $G$ is $P_k$-free, then $c(G) \leq k - 3$.
\end{conjecture}

Significantly for this paper, Liu \cite{liu2019cop} proved that all $(P_k,\claw)$-free graphs have cop number at most $k-3$, thereby satisfying Conjecture~\ref{Conj: main}.
However, proving Sivaraman's conjecture,
even for all $P_5$-free graphs,
was highly non-trivial.
Initially many authors, see \cite{gupta2023cops,liu2019cop,sivaraman2019cop,turcotte2022cops}, solved various subcases of the $P_5$-free case.
The general $P_5$-free conjecture was only recently proven
by
Chudnovsky, Norin, Seymour, and Turcotte in \cite{chudnovsky2024cops}.
Their proof is technical. The key step of which being to prove that every $P_5$-free graph with independence number at least $3$ contains a $3$-vertex induced path with vertices $abc$ in order, such that
every neighbour of $c$ is also adjacent to one of $a, b$.

Less work has been devoted to demonstrating the existence of $P_k$-free graphs with large cop number.
The $4$-cycle is a $P_4$-free graph with cop number $2$, so the Chudnovsky, Norin, Seymour, and Turcotte theorem
is tight.
The Petersen graph is a $P_6$-free graph with cop number $3$, so again Sivaraman's conjecture predicts a tight upper bound.
Interestingly, $C_4$ is the smallest graph with cop number $2$, and the smallest graph with cop number $3$ is the Petersen graph \cite{baird2014}.
The smallest graph with cop number $4$ is the Robertson graph \cite{turcotte2021}, which contains an induced $P_{11}$.
For $k\geq 7$ it is not-obvious that there exists a $P_k$-free graph with cop number $k-3$.

Another problem, proposed by Turcotte in \cite{turcotte2022cops}, was to determine if for $\ell\geq 3$ there exists graphs $G$ with $c(G) = \alpha(G) = \ell$.
The first progress on this problem comes from
Char, Maniya, and Pradhan \cite{char20254} who demonstrated a $16$ vertex graph with cop number and independence number $3$.
Significantly, such a graph is necessarily $P_7$-free.
More recently, Clow and Zaguia \cite{clow2025cops} prove that for all positive integers $\ell$ there is a graph with $c(G) = \alpha(G) = \ell$.
All such graphs are necessarily $P_{2\ell+1}$-free.
Hence, a corollary of Clow and Zaguia's result is that for all $k$, there exists $P_{k}$-free graphs with cop number at least $\lfloor\frac{k-1}{2}\rfloor$.
For all values of $k\geq 5$, 
this result provides a best known example of a graph $P_k$-free graph with large cop number.
Clow and Zaguia's construction uses random graphs of diameter $2$.
Connections between a graph's cop number and diameter has been extensively studied, see \cite{lu2012meyniel,petr2023note,wagner2015cops}.

Our contributions are as follows.
We begin by demonstrating a subclass of $P_k$-free graphs with small cop number.
This class is significant in light of our next result.

\begin{theorem}\label{Thm: path,claw,butterfly free}
    If $G$ is a $(P_k,\claw,\butterfly,C_4,C_5)$-free graph, then 
    \[
    c(G) \leq \Big\lceil \frac{k-1}{3}\Big\rceil+3.
    \] 
\end{theorem}

The choice of $\claw,\butterfly,C_4,C_5$ here is not arbitrary.
These graphs come from the following theorem, whose proof uses an operation appearing in \cite{joret2010cops}.

\begin{theorem}\label{Thm: path to claw reduction}
    If $G$ is a graph whose longest path is length $p$ and $c(G) \geq t$, then there exists a $(P_{2p+1},\claw,\butterfly,C_4,C_5)$-free graph $H$ with $c(H) \geq t$.
\end{theorem}

Together, Theorem~\ref{Thm: path,claw,butterfly free} and Theorem~\ref{Thm: path to claw reduction} imply the following result.
In \cite{joret2010cops} it was shown that if a graph $G$ has no cycle of length $p$, 
then $G$ has cop number at most $\frac{p}{2}$. 

\begin{theorem}\label{Thm: main longest path}
   If $G$ is a graph whose longest path is length $p$, then $c(G) \leq \lceil\frac{2p}{3}\rceil+3$. 
\end{theorem}

This theorem should be understood in the context of the 
weak Meyniel conjecture, 
which states that there exists an $\epsilon>0$ such that all $n$ vertex graphs,  $G$,
have $c(G) = O(n^{1-\epsilon})$.
Trivially, $p$ is at most $n$ in an $n$-vertex graph. If the weak Meyniel conjecture
were to be false, then there exist graphs whose longest path is length $p$ with cop number $c(G) = \Omega(p^{1-o(1)})$.
Do such graphs exist?

The final result we prove 
generalizes and strengthens Liu's result that $(P_k,\claw)$-free graphs
have cop number at most $k-3$.
Notice that every $\claw$-free graph is $\E$-free, since the $\claw$ is an induced subgraph of $\E$.

\begin{theorem}\label{Thm: E free}
    If $G$ is a $(P_k,\E \,)$-free graph, then 
    \[
    c(G) \leq \Big\lceil \frac{k-1}{2}\Big\rceil+3.
    \] 
\end{theorem}

The rest of the paper is structured as follows.
In Section~2 we define some terms that we introduce, 
and we recall the definition of clique substitution operation used in \cite{joret2010cops}.
Next, in Section~3 we prove Theorem~\ref{Thm: path,claw,butterfly free}.
The focus on Section~4 is to prove Theorem~\ref{Thm: path to claw reduction}, 
thereby completing the proof of Theorem~\ref{Thm: main longest path}.
In Section~5 we prove Theorem~\ref{Thm: E free}.
We conclude in Section~6 with a discussion of future work.

\section{Preliminaries}

In this section we define terms relevant for the rest of the paper.

For positive integers $k$ and $t$, and a set $S \subseteq [k] \times [t]$,
a graph $G$ is a $(k,t,S)$-flail 
if $G$ consists of an induced path $P:u_1u_2\dots u_{k+1}$ on $(k+1)$-vertices 
and $t$ vertices $v_1,\dots, v_t$ in $N(u_{k+1})$
such that $\{u_1,\dots, u_{k+1}\} \cap \{v_1,\dots, v_t\} = \emptyset$.
From here $(i,j) \in S$ if and only if $u_i$ and $v_j$ are adjacent.
See Figure~\ref{fig: FlailExample} for an example of a flail.
The edges between vertices in $v_1,\ldots,v_t$ are not specified.

\begin{figure}[h!]
    \centering
    \includegraphics[scale = 1.25]{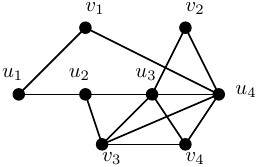}
    \caption{An example of a $(3,4,S)$-flail where $S = \{(1,1),(2,3),(3,2),(3,3),(3,4)\}$.}
    \label{fig: FlailExample}
\end{figure}

Let $P: u_1\dots u_k$ be a path and $X$ be a subset of the vertices $\{u_1,\dots, u_k\}$.
We say that $X$ $\frac{1}{3}$-saturates $P$ if for all $i$, 
$X \cap \{u_i,u_{i+1}\} \neq \emptyset$ or $u_{i-1},u_{i+2}\in X$.
Notice the second clause cannot be fulfilled if $i = 1$, so $X$ being $\frac{1}{3}$-saturating
implies $u_1$ or $u_2$ is in $X$. Intuitively, every consecutive set of three vertices has at least one vertex in $X$, and either $u_1$ or $u_2$ is also in $X$.

Next, we introduce the operation of clique substitution used in \cite{joret2010cops}.
Let $G = (V,E)$ be a graph, the clique substitution $H$ of $G$ is defined as follows.
For all $v \in V$, let $K^v$ be a clique in $H$ with vertices $\{(v,u): u \in N(v)\}$.
Hence, for all vertices $u\neq v$, $K^v$ and $K^u$ are vertex disjoint.
From here we complete the definition of $H$, by adding the edge
$(v,u)(u,v)$  between cliques $K^u$ and $K^v$ in $H$ if $uv \in E$.
See Figure~\ref{fig: clique substitutionExample} for an example of a clique substitution.

\begin{figure}[h!]
    \centering
    \includegraphics[scale = 1.00]{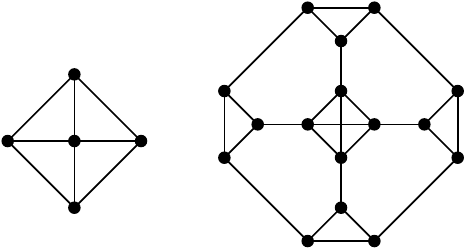}
    \caption{The $4$-wheel is drawn on the left and the clique substitution of the $4$-wheel is drawn on the right.}
    \label{fig: clique substitutionExample}
\end{figure}

\section{$(P_k,\claw,\butterfly,C_4,C_5)$-free Graphs}

In this section we describe a strategy by which $\lceil \frac{k-1}{3}\rceil+3$ cops can capture the robber
on a $(P_k,\claw,\butterfly,C_4,C_5)$-free graph.
In doing so, we will show that in any $(P_k,C_4,C_5)$-free graph,
$\lceil \frac{k-1}{3}\rceil+3$ cops have a strategy which ensures the robber's only winning move relies on the existence of certain induced subgraphs.
Then, we use the fact that our graphs are $(\claw,\butterfly)$-free to restrict the set of these subgraph that can exist.
This limits the robber sufficiently for the cops to capture.

To start we provide some lemmas regarding which $(k,t,S)$-flails may be induced subgraphs of $(\claw,\butterfly)$-free graphs.

\begin{lemma}\label{Lemma: Flail claw Lemma}
    Let $G$ be a $\claw$-free graph and $H$ an induced subgraph of $G$.
    If $H$ is a $(k,t,S)$-flail
    such that $k\geq 3$, then for all $(i,j) \in S$ where $1< i < k$, $(i-1,j)\in S$ or $(i+1,j) \in S$.
\end{lemma}

\begin{proof}
    Let $G$ be a $\claw$-free graph and $H$ an induced subgraph of $G$.
    For contradiction suppose $k\geq 3$ and $1 < i < k$, such that 
    $(i,j) \in S$, while $(i-1,j)\notin S$ and $(i+1,j) \notin S$.
    Then $u_{i-1}u_i,u_iu_{i+1},u_iv_j$ are all edges of $H$. 
    Meanwhile $u_{i-1}u_{i+1}$ is a non-edge in $H$, since $u_1\dots u_{k+1}$ induces a path.
    By our assumption $(i-1,j)\notin S$ and $(i+1,j) \notin S$, we note $u_{i-1}v_j$ and $u_{i+1}v_j$ 
    are also non-edges.
    Thus, $\{u_{i-1},u_i,u_{i+1},v_j\}$ induces a $\claw$ in $H$.
    Since $H$ is an induced subgraph of $G$, $\{u_{i-1},u_i,u_{i+1},v_j\}$ induces a $\claw$ in $G$.
    This contradicts $G$ being $\claw$-free, thereby completing the proof.
\end{proof}

\begin{figure}[h!]
    \centering
    \includegraphics[scale = 1.05]{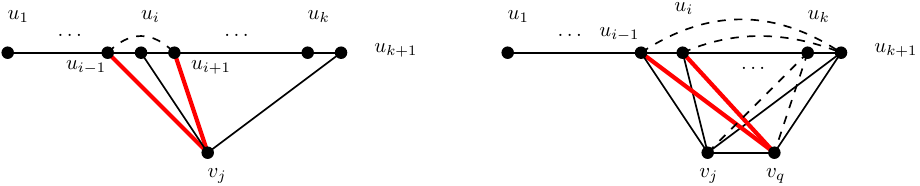}
    \caption{
    The left figure depicts the flail discussed in Lemma~\ref{Lemma: Flail claw Lemma}, while the right figure depicts the flail discussed in Lemma~\ref{Lemma: Flail claw,butterfly Lemma}. Both lemmas claim at least one of a certain set of edges must exist. In both figures these edges are drawn as red and bold.}
    \label{fig:FlailLemmas}
\end{figure}

\begin{lemma}\label{Lemma: Flail claw,butterfly Lemma}
    Let $G$ be a $(\claw,\butterfly)$-free graph and $H$ an induced subgraph of $G$.
    If $H$ is a $(k,t,S)$-flail
    such that $k\geq 3$, and
    \[\{ (k,j): j \in [t]\} \cap S = \emptyset,\]
    and $(i,j),(i-1,j)\in S$ for some $i$ and $j$, then for all $q \in [t]$, $(i,q)\in S$ or $(i-1,q)\in S$.
\end{lemma}

\begin{proof}
    Let $k\geq 3$, 
    let $G$ be a $(\claw,\butterfly)$-free graph, and let $H$ an induced $(k,t,S)$-flail in $G$.
    For contradiction suppose
    \[\{ (k,j): j \in [t]\} \cap S = \emptyset,\]
    and $(i,j),(i-1,j)\in S$ for some $i$ and $j$ but there exists a $q \in [t]$ such that $(i,q),(i-1,q)\notin S$.
    Then $q\neq j$ implying $t \geq 2$.

    Since $G$ is $\claw$-free, $H$ is $\claw$-free. 
    Hence, $\{ (k,j): j \in [t]\} \cap S = \emptyset$ implies that 
    vertices $\{v_1,\dots, v_t\}$ induce a clique.
    Otherwise, since $t\geq 2$, there are distinct and non-adjacent vertices $v_a,v_b$, 
    implying the vertices $\{u_k,u_{k+1},v_a,v_b\}$ induce a $\claw$, which is a contradiction.
    Suppose then that
    vertices $\{v_1,\dots, v_t\}$ induce a clique.

    Since
    vertices $\{v_1,\dots, v_t\}$ induce a clique, vertices $v_q$ and $v_j$ are adjacent.
    Trivially, $v_q$ and $v_j$ are both adjacent to $u_{k+1}$.
    By our assumption that $(i,q),(i-1,q)\notin S$, $v_q$ is not adjacent to $u_i$ or $u_{i-1}$.
    Given we assumed  $(i,j),(i-1,j)\in S$, we note that $v_j$ is adjacent to $u_i$ and $u_{i-1}$.
    Finally, we note that since $u_1\dots u_{k+1}$ is an induced path by the definition of being a $(k,t,S)$-flail,
    vertices $u_i$ and $u_{i-1}$ are adjacent, while $i < k$ since $(k,j)\notin S$, implies that both $u_i$ and $u_{i-1}$ are non-adjacent to $u_{k+1}$.
    But this implies vertices $\{u_{i-1},u_i,u_{k+1},v_j,v_q\}$ induces a $\butterfly$.
    Since $H$ is an induced subgraph of $G$, which is $\butterfly$-free, this is a contradiction.
    Thereby completing the proof.
\end{proof}

We are now prepared to prove Theorem~\ref{Thm: path,claw,butterfly free}.
We will employ the Gy\'arf\'as path argument in a similar manner to Sivaraman in \cite{sivaraman2019application}.

\begin{proof}[Proof of Theorem~\ref{Thm: path,claw,butterfly free}]

Let $G$ be a $(P_k,\claw,\butterfly,C_4,C_5)$-free graph.
We will show how  $\lceil\frac{k-1}{3}\rceil+3$ cops can capture the robber, no matter how the robber plays.
The cops begin the game with all $\lceil\frac{k-1}{3}\rceil+3$ cops on a fixed but arbitrary vertex $w_0$.
Let $v_0$ denote the starting position of the robber.
Label the cops $C^\uparrow,C^\downarrow, C^\Downarrow, C^0,\dots, C^{\lceil\frac{k-1}{3}\rceil-1}$.
We denote the robber by $R$.

If $\dist(w_0,v_0) = 1$, then the cops capture on their first turn.
Assume without loss of generality that the robber never deliberately moves adjacent to a cop, as this is losing for the robber.
Then $\dist(w_0,v_0)\geq 2$.
If $\dist(w_0,v_0) = 2$, then proceed to Step~2 of the cops' strategy.
Otherwise, $\dist(w_0,v_0) > 2$ in which case proceed to Step~1 of the cops' strategy.

\vspace{0.25cm}
\noindent\underline{Step.1:}  We suppose $\dist(w_0,v_0) > 2$, all cops are on $w_0$, the robber is on $v_0$, and it is the cops turn to move.
\vspace{0.25cm}

We aim to show that the cops can ensure that after finitely many turns, the robber is at distance at most $2$ from some cop on the cops' turn to move.
For contradiction suppose the robber has and is using a strategy which ensures they remain at distance at least $3$ from the cops before the cops move.

The cops proceed as follows.
For $i\geq 0$ we let $v_i$ denote the position of robber before the cops move in turn $i+1$, 
and let $w_i$ be the location of the cop $C^0$ before the cops move in turn $i+1$.
Hence, $w_0w_1\dots w_i$ is a walk for all $i\geq 0$.
Cops $C^\downarrow$ and $C^{\Downarrow}$ never leave $w_0$, 
while cop $C^\uparrow$ always moves with cop $C^0$, that is $C^\uparrow$ always moves to $w_i$ on turn $i$. 
Hence, before the cops move on turn $i+1$ there are at least two cops on $w_0$ and at least two cops on $w_i$.
Let $x$ be an arbitrary positive integer, we let $w_{-x} = w_0$.
Then for all $j>0$ we suppose that the cop $C^j$ occupies vertex
$w_{i-3j}$ prior to the cops move in turn $i+1$. 
That is, initially all cops remain on $w_0$,
and as the distance from $w_0$ to $w_i$ increases, another cop leaves $w_0$ to follow the walk taken by $C^0$.
This is possible since $w_0w_1\dots w_i$ is a walk and all cops begin on $w_0$.
Let 
\[
D_i = \min\Big(\Big\{\dist(w_0,v_i)\Big\} \cup  \Big\{\dist(w_{i-3j},v_i): 0 \leq j \leq \Big\lceil\frac{k-1}{3} \Big\rceil-1\Big\} \Big).
\]
Notice this is the smallest distance from any cop to the robber before the cops move in turn $i+1$.
We now describe the movement of $C^0$.

Let $G_0 = G$.
On turn $1$, the cop $C^0$ is on $w_0$.
Since $D_0 > 2$ and $G$ is connected, there exists a vertex $u \in N[w_0]$ such that $2 \leq \dist(u,v_0) < D_0$.
Cop $C^0$ moves to such a vertex $u$ which sets $w_1 = u$.
The robber moves from $v_0$ to $v_1$.
Let $G_1 = G_0 - (N[w_0]\setminus \{w_1\})$.
By assumption the robber moves so that $D_1 \geq 3$.
If $v_1 \in N[w_0]$, then the robber is at distance $2$ from $w_1$ a contradiction.
Thus, $v_1 \in V(G_1)$.
Then, $2 < D_1 \leq \dist_{G_1}(w_1,v_1) < \infty $. This completes the base case of the following induction.

Let $i < k-2$.
Suppose that for all $1 \leq j \leq i$
\[2 < D_j \leq \dist_{G_j}(w_j,v_j) < \infty,\]
and $G_j = G_{j-1} - (N[w_{j-1}]\setminus \{w_j\})$.
Consider the game in turn $i+1$.
Since $D_i \leq \dist_{G_i}(w_i,v_i)$ which is finite, 
$v_i \in V(G_i)$ and there exists a vertex $u \in N_{G_i}[w_i]$ 
such that there is a path from $u$ to $v_i$ in $G_i$ which is vertex disjoint from $N[w_i]\setminus \{u\}$.
The cop $C^0$ moves from $w_i$ to $u$, which sets $w_{i+1} = u$.
The robber then follows their strategy to their next vertex $v_{i+1}$. Hence, $D_{i+1} \geq 3$.

Let $G_{i+1} = G_i - (N[w_i]\setminus \{w_{i+1}\})$.
Since $i < k-2$ and the cops occupy vertices 
\[
\Big \{w_{i+1-3j}: 0 \leq j \leq \left\lceil\frac{k-1}{3} \right\rceil-1 \Big\} \cup \{w_0\}
\]
the cops occupy a dominating set of the walk $w_0\dots w_{i+1}$.
Hence, $D_{i+1}\geq 3$ implies that $v_{i+1}$ is not adjacent to any vertex $w_j$.
That is,
$v_{i+1} \notin \cup_{j=0}^{i+1} N[w_j]$,
which implies that $v_{i+1} \in V(G_{i+1})$.
Thus, $2 < D_{i+1} \leq \dist_{G_{i+1}}(w_{i+1},v_{i+1}) < \infty$. So our induction is sufficient for all
$i \leq k-2$.

Consider the game before the cops move in turn $k-1$.
The robber occupies vertex $v_{k-2}$ in $G_{k-2}$, hence the graph $G_{k-2}$ is non-empty.
Moreover,
$D_{k-2} \leq \dist_{G_{k-2}}(w_{k-2},v_{k-2}) < \infty$
implies that $w_{k-2}$ has a neighbour $u$ in $G_{k-2}$.
By the definition of our vertices $w_i$ and our graphs $G_i$, we note that $w_0w_1\dots w_{k-2}u$ is an induced path of length $k$ in $G$.
This contradicts the fact that $G$ is $P_k$-free.

We conclude that for some $0 \leq i \leq k-2$ it must the case that $D_{i}\leq 2$.
Suppose without loss of generality that $t\geq 0$ is the least integer such that $D_t \leq 2$.
Since the cops are about to move when the distance $D_t$ is computed, the robber loses if $D_t \leq 1$.
Suppose then that $D_t = 2$.
On turn $t+1$ the cops proceed to Step 2.

\vspace{0.25cm}
\noindent\underline{Step.2:} It is the cops turn, vertices $w_0\dots w_t$ form an induced path, each cop $C^j$ occupies vertex $w_{t-3j}$, 
where $w_{-x} = w_0$ for any positive $x$, 
cop $C^\uparrow$ occupies vertex $w_t$,
cops $C^\downarrow$ and $C^\Downarrow$ occupy vertex $w_0$,
and there exists a cop $C$ such that $\dist(C,R) = 2$.
\vspace{0.25cm}

If $\dist(w_0,v_t) = 2$, then let $C = C^{\Downarrow}$. 
Else, if $\dist(w_t,v_t) = 2$ let $C^0 = C$.
In all other cases if there are multiple cops at distance $2$ from $R$, then choose $C$ to be a fixed but arbitrary one of these cops.
Let $u_0$ denote the location of the cop $C$ at distance $2$ from the robber $R$.

We proceed in a similar manner to Step 1, 
but we must handle the transition of cops from the path $w_0\dots w_t$ to a new path $u_0\dots u_\ell$ carefully.
If $C^i = C$ for some $i>0$, then let $M = \min\{i, \lceil \frac{k-1}{3} \rceil-1-i\}$.
Otherwise, let $M = 0$, when $C$ is either $C^\Downarrow$ or $C=C^0$.
Thus, $M$ measures how close the cop $C$ is to an end of the path $w_0\dots w_t$.
We relabel the cops as follows, 
\begin{itemize}
    \item Let $\mathcal{C}^0 = C$, and
    \item Let $\mathcal{C}^\uparrow = C^\uparrow$ and $\mathcal{C}^\downarrow = C^\downarrow$  and
    \item if $M>0$ and $C^i = C$, then for all $0 < j \leq M$ we let $\mathcal{C}^{2j-1} = C^{i-j}$ and $\mathcal{C}^{2j} = C^{i+j}$, and
    \item if $M\geq 0$ and $C^i = C$, then for $j > M$ we let $\mathcal{C}^{2M+j} = C^{i\pm j}$, and
    \item if $C^\Downarrow = C$, then for all $j\geq 1$, $\mathcal{C}^j = C^{\lceil \frac{k-1}{3} \rceil  -j}$, and
\end{itemize}
Notice that for $j>M$, at most one of $C^{i-j}$ or $C^{i+j}$ exists so this labelling is well defined.

We reset the turn counter to $0$ to avoid any confusion.
For $i\geq 0$ we let $r_i$ denote the position of robber before the cops move in turn $i+1$, 
and let $u_i$ be the location of the cop $\mathcal{C}^0$ before the cops move in turn $i+1$.
For $i\geq 0$, 
let 
\[
d_i = \min\Big(\Big\{\dist(u_0,r_i)\Big\} \cup  \Big\{\dist(u_{i-3j},r_i): 0 \leq j \leq \left\lceil\frac{k-1}{3} \right\rceil-1\Big\} \Big).
\]
As in Step 1 we let $u_{-x} = u_0$  when $x$ is a positive integer.
Then $d_0 = 2$.

As in Step 1, we describe how cops other than 
$\mathcal{C}^0$ move in terms of the previous moves of $\mathcal{C}^0$.
Then, we describe the movement of $\mathcal{C}^0$.
However, unlike in Step 1, 
we also describe conditions under which all cops break from this strategy in order to capture the robber.

All cops $\mathcal{C}$ other than $\mathcal{C}^0$ begin by proceeding to $u_0$ along the path $w_0\dots w_t$.
For all $j>0$, since each cop $\mathcal{C}^j$ has distance at most $3j$ to $u_0$ along the path $w_0\dots w_t$, 
cop $\mathcal{C}^j$ reaches $u_0$ on turn $3j$ or sooner.
If $\mathcal{C}^j$ reaches $u_0$ prior to turn $3j$, they remain on vertex $u_0$ until turn $3j$.
So for all $j>0$, before moving in turn $3j+1$ cop $\mathcal{C}^j$ occupies vertex $u_0$,
and on turn $3j+1$ cop $\mathcal{C}^j$
begins to move along the walk $u_0\dots u_i$.
So for all $j\geq 0$, if $i\geq 3j$, then cop $\mathcal{C}^j$ occupies vertex $u_{i-3j}$ before moving on turn $i+1$.
Once cop $\mathcal{C}^\uparrow$ or $\mathcal{C}^\downarrow$ reaches $u_0$, they remain at $u_0$ for the rest of play.
See Figure~\ref{fig:w-to-u} for some assistance visualizing this movement.

\begin{figure}[h!]
    \centering
    \includegraphics[scale = 0.75]{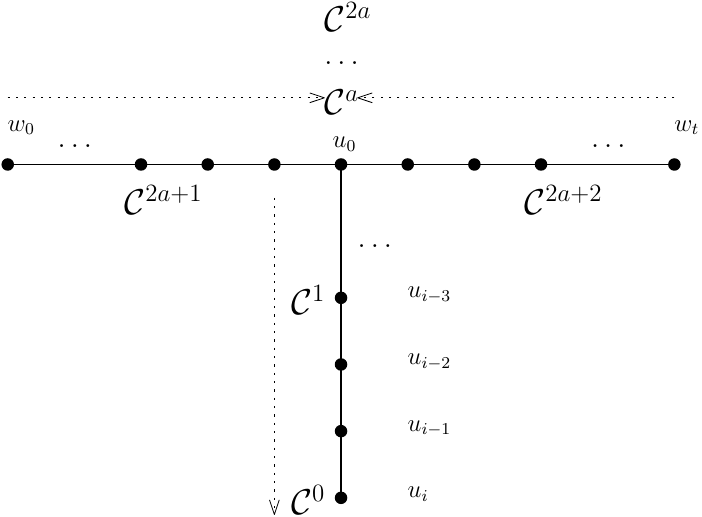}
    \caption{An image of how cops other than $\mathcal{C}^0$ move on turn $i=3a < M$. Here the walk $u_0\dots u_i$ need not be internally vertex disjoint from the path $w_0\dots w_t$.}
    \label{fig:w-to-u}
\end{figure}
\noindent\underline{Claim 1:} For all $i\geq 3$ before the cops move on turn $i+1$, 
there is a cop $\mathcal{C}$ occupying $u_0$ which remains on $u_0$ in turn $i+1$.
\vspace{0.25cm}

We deal with cases separately depending on which vertex was selected as $\mathcal{C}^0$.
If $\mathcal{C}^0 = C^{\Downarrow}$, then cop $\mathcal{C}^\downarrow$ begins on $u_0$ 
before the cops move in turn $1$, then never moves again. So on all turn $i\geq 3$, $\mathcal{C}^\downarrow$ will be a cop on $u_0$.
Similarly, if $\mathcal{C}^0 = C^{0}$, then $\mathcal{C}^\uparrow = C^{\uparrow}$ begins on $u_0$ 
and will remain there for all turns $i\geq 3$.

Otherwise $\mathcal{C}^0 = C^{\ell}$ for some $\ell>0$.
In this case $M>0$.
So for all $0 < j \leq M$ there is a cop $\mathcal{C}^{2j-1}$ and a cop $\mathcal{C}^{2j}$, both of whom will arrive at vertex $u_0$ after $3j$ turns or earlier.
Notice that once cop $\mathcal{C}^{2j-1}$ moves to $u_1$, cop $\mathcal{C}^{2j}$ will remain on $u_0$ for $3$ more turns. 
If $j < M$, this allows there to be enough time for cops $\mathcal{C}^{2j+1}$ and $\mathcal{C}^{2j+2}$ to arrive at $u_0$ before $\mathcal{C}^{2j}$ leaves $u_0$.
If $j = M$, then before $\mathcal{C}^{2j}$ leaves $u_0$ cop $\mathcal{C}^\uparrow$ or $\mathcal{C}^\downarrow$ will arrive at $u_0$.
Once $\mathcal{C}^\uparrow$ or $\mathcal{C}^\downarrow$ arrives there will always be a cop on $u_0$ for the rest of the game.

Therefore, for all $i\geq 3$, there will be a cop on $u_0$ before the cops move on turn $i+1$ which remains on $u_0$ in turn $i+1$.
\hfill $\diamond$
\vspace{0.5cm}

\vspace{0.25cm}
\noindent\underline{Claim 2:} For all $i\geq 3$, if $P: u_0\dots u_i$ is a path, then the cops on $P$ $\frac{1}{3}$-saturate $P$.
\vspace{0.25cm}

Let $i\geq 3$ and suppose $P: u_0\dots u_i$.
By the instructions given to cops $\mathcal{C}^j$ for all $j$ such that $i\geq 3j$, the cop $\mathcal{C}^j$ is located on vertex $u_{i-3j}$.
Meanwhile, Claim 1 implies there is always a cop on vertex $u_0$ when $i\geq 3$.
Hence, for any consecutive pair of vertices $u_{\ell}u_{\ell+1}$ on the path $P$ either there is a cop on one, or both, of these vertices,
or there is a cop on $u_{\ell-1}$ and another cop on $u_{\ell+2}$.
\hfill $\diamond$
\vspace{0.5cm}

Now we describe the movement of the cop $\mathcal{C}^0$.
In turn $1$ the cop $\mathcal{C}^0$ moves as follows.
Let $H_0 = G$.
Since $\dist(u_0,r_0) = 2$, there is a vertex $x \in N(u_0) \cap N(r_0)$.
The cop $\mathcal{C}^0$ moves to $x$, which sets $u_1 = x$.
In response, the robber moves from $r_0$ to some vertex $r_1 \notin N[u_1]$.
Since $u_0u_1r_0$ and $u_1r_0r_1$ are both induced paths, 
the fact that $G$ is $C_4$-free implies that $u_0$ and $r_1$ are non-adjacent.
Let $H_1 = H_0 - (N[u_0]\setminus \{u_1\})$.
It follows that $\dist_{H_1}(u_1,r_1) = 2$, and $r_1$ is non-adjacent to any vertex in $\{u_0,u_1\}$.

Let $i < k-1$.
Suppose that for all $1 \leq j \leq i$, $\dist_{H_j}(u_j,r_{j-1}) = 1$ and $\dist_{H_j}(u_j,r_{j}) = 2$,
while $H_j = H_{j-1} - (N[u_{j-1}]\setminus \{u_j\})$.
By the above, when $i=1$, the base case holds.
Consider the game in turn $i+1$.

In turn $i+1$,
the cop $\mathcal{C}^0$ moves from $u_i$ to $r_{i-1}$.
This sets $u_{i+1} = r_{i-1}$.
By the induction hypothesis $u_{i+1} = r_{i-1}$ is a vertex of $H_i$.
In response the robber moves from $r_{i}$ to some vertex $r_{i+1} \notin N[u_{i+1}]$.

\vspace{0.25cm}
\noindent\underline{Claim 3:} $r_{i+1}$ is not adjacent to $u_i$ or $u_{i-1}$.
\vspace{0.25cm}

By the induction hypothesis $r_{i-1}$ and $r_i$ are both vertices of $H_i$.
Hence, by the definition of $H_i$,
\[
\{r_{i-1},r_i\} \cap \left(\bigcup_{j=0}^{i-1} N[u_j] \right) = \emptyset.
\]
Thus, $u_{i-1}u_{i}r_{i-1}r_i$ is an induced path.
Furthermore, we know $u_{i+1} = r_{i-1}$ and $r_{i+1}$ are not adjacent.
Suppose $u_{i}$ and $r_{i+1}$ are adjacent, then $u_ir_{i-1}r_ir_{i+1}$ induces $C_4$.
Since this is a contradiction we note that $u_{i}$ and $r_{i+1}$ are non-adjacent.
Suppose now that  $u_{i-1}$ and $r_{i+1}$ are adjacent, then $u_{i-1}u_ir_{i-1}r_ir_{i+1}$ induces $C_5$.
Since this is a contradiction we note that $u_{i-1}$ and $r_{i+1}$ are non-adjacent.
\hfill $\diamond$
\vspace{0.5cm}

\vspace{0.25cm}
\noindent\underline{Claim 4:} If $r_{i+1}$ has a neighbour in $u_0\dots u_{i+1}$, 
then the cops have a strategy to capture the robber. 
\vspace{0.25cm}

Recall $r_{i+1} \notin N[u_{i+1}]$.
If $i=1$, then by Claim 3 $r_{i+1} = r_2$ is not adjacent to $u_{i-1} = u_0$ or $u_{i}= u_1$.
So the claim is vacuously true when $i = 1$.
We consider the cases $i = 2$, and $i>2$ separately.

If $i = 2$, then Claim 3 implies $r_{i+1} = r_3$ is not adjacent to $u_{i-1} = u_1$ or $u_{i}= u_2$.
Additionally, Claim 1 implies that a cop occupies $u_0$ at the start of turn $4$.
Call this cop $\mathcal{C}$.
Recalling that $r_{i+1} = r_3$ is the location of the robber prior to the cops move in turn $4$,
if $r_3$ is adjacent to $u_0$, then cop $\mathcal{C}$ can capture the robber on turn $4$.
This requires $\mathcal{C}$ to break from their normal behaviour,
but this has no effect on any other cases, because $\mathcal{C}$ captures the robber immediately.
This concludes the $i=2$ case.

Now suppose that $i\geq 3$ and that $r_{i+1}$ has a neighbour in $u_0\dots u_{i+1}$.
If $r_{i+1}$ is adjacent to $u_{i+1}$, then cop $\mathcal{C}^0$ will capture in turn $i+2$.
Suppose then without loss of generality that $r_{i+1}$ is not a neighbour of $u_{i+1}$ and $r_{i+1}$ is adjacent to some vertex in  $u_0\dots u_{i}$.
Thus, there is a vertex $v \in N[r_i]\setminus N[u_{i+1}]$ with neighbours on  $u_0\dots u_i$.
Let $v$ be any such vertex,
and suppose without loss of generality
that $0 \leq \ell \leq i$ such that $u_\ell \in N(v)$.

Consider the cops' move on turn $i+1$, that is the cops' move prior to the robber's move from $r_i$ to $r_{i+1}$.
Notice that the cops will not know which vertex the robber selects as $r_{i+1}$, 
however, this has no impact on the following argument.
The cop $\mathcal{C}^0$ moves as normal, while other cops break from their normal behaviour to capture the robber.
The signal for this switch in behaviour by the cops is the existence of a vertex 
$v \in N[r_i]\setminus N[r_{i-1}]$ with neighbours on  $u_0\dots u_i$.
The cops can check if there exists a vertex $v$
without knowledge of where the robber will move in turn $i+1$.

If a cop $\mathcal{C}$ is standing on vertex $u_\ell$, or any other vertex adjacent to $v$,
before the cops move on turn $i+1$,
then $\mathcal{C}$ can move to $v$ at the same time the cop $\mathcal{C}^0$ moves to $r_{i-1}$.
Recall that $v$ and $u_{i+1} = r_{i-1}$ are non-adjacent.
Since $G$ is $\claw$-free the independence number of $G[N[r_i]]$ is at most $2$.
Since  $v$ and $u_{i+1} = r_{i-1}$ are non-adjacent and in the neighbourhood of $r_i$, $\{v,u_{i+1}\}$ forms a dominating set of $N[r_i]$. 
So every neighbour of $r_i$ is adjacent to either $r_{i-1}$ or $v$, that is,
\[
N[r_i] \subseteq N[v] \cup N[r_{i-1}].
\]
This implies that one of the cops $\mathcal{C}$ and $\mathcal{C}^0$, which are distinct, will capture the robber on the next cop turn, regardless of where the robber chooses to move.

Otherwise, no cop is standing adjacent to $v$.
By Claim 1 there is a cop on $u_0$ before the cops move in turn $i+1$,
so $\ell> 0$,
while Claim 3 implies that $\ell \leq i-2$.
By the definition of the walk $P: u_0\dots u_{i+1}r_i$ and the graphs $H_{i}$, $P$ is an induced path.
Thus, $F = G[\{u_0,\dots, u_{i+1},r_i\} \cup ( N[r_{i}]\setminus N[u_{i+1}])]$ 
is a $(i+2,|N[r_{i}]\setminus N[u_{i+1}]|, S)$-flail where $S$ is determined by the adjacencies of 
$\{u_0,\dots, u_{i+1}\}$ with $N[r_{i}]\setminus N[u_{i+1}]$.

Since $G$ is $(\claw,\butterfly)$-free, $F$ is $(\claw,\butterfly)$-free.
Then $0 < \ell \leq i-2$ and
Lemma~\ref{Lemma: Flail claw Lemma} implies that 
$v$ must be adjacent to $u_{\ell-1}$ or $u_{\ell+1}$.
We recall that 
the instructions provided to cops $\mathcal{C}^j$ ensure that there is a cop on each vertex 
$\{u_{i-3j}: 0 \leq j \leq \lceil \frac{k-1}{3}\rceil-1\}$ where $u_{-x} = u_0$ for positive integers $x$.

Since $i\geq 3$, there is a cop on vertex $u_{i-3}$.
If $\ell = i-2$, then $u_{i-3}$ or $u_{i-1}$ is adjacent to $v$.
We have already shown $v$ and  $u_{i-1}$ being adjacent contradicts $G$ being $(C_4,C_5)$-free.
So $v$ must be adjacent to $u_{i-3}$.
But there is a cop on $u_{i-3}$, contradicting that  no cop is standing adjacent to $v$ before the cops move on turn $i+1$.
So $\ell \leq i-4$, since $u_{i-3}$ and $u_{i-2}$ are both non-adjacent to $v$.

Without loss of generality and by Lemma 3.1, suppose that $v$ is adjacent to $u_\ell$ and $u_{\ell-1}$, 
neither of which contains a cop before the cops move on turn $i+1$.
Since $F$ is $(\claw,\butterfly)$-free, Lemma~\ref{Lemma: Flail claw,butterfly Lemma} implies 
every vertex in $N[r_i]\setminus N[u_{i+1}]$ is adjacent to either $u_\ell$ or $u_{\ell-1}$.
By Claim 2, before moving in turn $i+1$ the cops $\frac{1}{3}$-saturate $u_0\dots u_i$. Since there is no cop on $u_\ell$ and $u_{\ell-1}$, 
this implies there is a cop $\mathcal{C}^A$ on $u_{\ell-2}$ and a cop $\mathcal{C}^B$ on $u_{\ell+1}$.
Given $\ell \leq i-4$ we conclude that both  $\mathcal{C}^A$ and  $\mathcal{C}^B$ are distinct from  $\mathcal{C}^0$.
Thus, cop $\mathcal{C}^A$ can move to $u_{\ell-1}$, and cop $\mathcal{C}^B$ can move to $u_{\ell}$ 
at the same time that cop $\mathcal{C}^0$ moves to $r_{i-1}$.
We have established 
\[
N[r_i] \subseteq N[u_{\ell-1}] \cup N[u_{\ell}]  \cup N[r_{i-1}]
\]
implying the cops can capture the robber on the next cop turn.
This concludes the proof of the claim.
\hfill $\diamond$
\vspace{0.5cm}

Suppose then that $r_{i+1}$ has no neighbour in $u_0\dots u_{i+1}$.
Then, letting 
\[
H_{i+1} = H_i - (N[u_{i}]\setminus \{u_{i+1}\}),
\] we note that 
$\dist_{H_{i+1}}(u_{i+1},r_{i}) = 1$ and $\dist_{H_{i+1}}\left(u_{i+1},r_{i+1}\right) = 2$.
So the induction hypothesis is maintained for another turn.
Next we show the induction hypothesis cannot be maintained for an infinite sequence of turns.

By the procedure already defined, $i\geq 1$.
Suppose for contradiction that $i\geq k-3$.
By the induction hypothesis and the definitions of the graphs $H_j$, the walk $P: u_0\dots u_ir_{i-1} r_i$ 
is an induced path of length at least $k$.
This contradicts the fact that $G$ is $P_k$-free.

Thus, after
finitely many turns
$r_{i+1}$ has a neighbour in $u_0\dots u_{i+1}$.
When this occurs Claim 4 proves the cops can capture the robber.
Therefore, $\lceil\frac{k-1}{3}\rceil+3$ cops have a strategy for catching the robber.
This completes the proof.
\end{proof}

In summary,
as is standard in Cops and Robbers, the cops can grow their territory to include the entire graph. 
Interestingly, and unlike many other cop territory arguments, 
the cops' territory is not monotone increasing. 
Instead, after some number of turns, the cops abandon their initial territory
in order to rebuild a new territory that will eventually become the entire graph.
This complexity is required, since the cops must begin at distance at most $2$ from the robber 
for their territory to include the entire graph.

\section{Clique Substitution of Graphs with Forbidden Path Subgraphs}

In this section we consider properties of the clique substitution  operation in order to prove Theorem~\ref{Thm: path to claw reduction}.
The first lemma is implicitly stated in \cite{joret2010cops}, as it is easy to see. We provide a proof for completeness because it is a key observation for the proof of Theorem~\ref{Thm: path to claw reduction}.

\begin{lemma}\label{Lemma: Clique-Sub Neighbourhood}
    If $H$ is the clique substitution  of $G$, then for all vertices $u$ in $G$ and $(u,v) \in V(K^u)$ in $H$, 
    $N_H((u,v))$ induces two disjoint cliques, one of size $\deg(u) - 1$ and the other size $1$.
\end{lemma}

\begin{proof}
    Let $H$ be the clique substitution  of $G$.
    Without loss of generality 
    let $u$ be a vertex of $G$ and $(u,v)$ a vertex in $K^u$.
    Since $v \in N(u)$ in $G$, there is an edge $(u,v)(v,u)$ in $H$.
    All other edges incident to $(u,v)$ are edges of $K^u$, and the size of $K^u$ is $|N(u)|$.
    Thus, the neighbourhood $N_H((u,v))$ consists of $(v,u)$ and the vertices of $K^u - (u,v)$. 
    There is at most $1$ edge between any pair of cliques $K^u$ and $K^v$ in $H$,
    so there is no edge $(u,x)(v,u)$.
    This completes the proof.
\end{proof}

Next, we consider some graphs that cannot appear in clique substitutions as induced subgraphs.
We are especially interested in the graphs forbidden in Theorem~\ref{Thm: path,claw,butterfly free}.

\begin{lemma}\label{Lemma: Clique-Sub C_4,C_5 free}
    If $H$ is the clique substitution  of $G$, then $H$ is $(C_4,C_5)$-free.
\end{lemma}

\begin{proof}
    Let $H$ be the clique substitution  of $G$.
    From the definition of $H$ we may $2$ colour the edges of $H$ into blue edges, which appear inside of a clique $K^u$ 
    for $u \in V(G)$, and red edges which are incident to two distinct cliques $K^u$ and $K^v$.
    By Lemma~\ref{Lemma: Clique-Sub Neighbourhood} each vertex in $H$ is incident to exactly $1$ red edge.
    
    Suppose for contradiction that $H$ contains an induced $C_4$, with vertices $abcd$.
    Since $a,c$ and $b,d$ are not adjacent, $a$ and $c$ belong to different cliques $K^u$ and $K^v$,
    while $b$ and $d$ also belong to distinct cliques $K^w$ and $K^z$.
    Here $u\neq v$ and $w \neq z$, but we make no claims about how $u$ relates to $w,z$ or how $v$ relates to $w,z$.
    
    Thus, at least one of the edges $ab$ or $bc$ is red.
    Since $b$ is incident to exactly one red edge, we suppose without loss of generality that $ab$ is red and $bc$ is blue.
    Then $w = v$.
    Since $a$ is incident to exactly one red edge, the edge $ad$ is blue, implying $u = z$.
    This implies $cd$ is a red edge since $u \neq v$.
    But there is at most one red edge between the cliques $K^u$ and $K^v$ contradicting that $ab$ and $cd$ are both red edges.
    It follows that $G$ is $C_4$-free.

    Proving $G$ is $C_5$-free is faster.
    Suppose $Q$ is an induced $C_5$ in $H$.
    Since $C_5$ is triangle-free no triple of vertices in $Q$ belong to a single clique $K^u$.
    Hence, each vertex in $Q$ is incident to at most $1$ blue edge in $Q$.
    Recalling that all edges in $H$ are incident to $1$ red edge, this implies that $Q$ can be properly $2$-edge-coloured.
    Since $Q$ is an induced $C_5$ this is a contradiction since $\chi'(C_5)= 3$.
\end{proof}

\begin{lemma}\label{Lemma: Clique-Sub induced paths}
    If $G$ is a graph whose longest path is length $p$ and $H$ is the clique substitution  of $G$, then $H$ is $P_{2p+1}$-free.
\end{lemma}

\begin{proof}
    Let $H$ be the clique substitution  of $G$.
    From the definition of $H$ we may $2$ colour the edges of $H$ into blue edges, which appear inside of a clique $K^u$ 
    for $u \in V(G)$, and red edges which are incident to two distinct cliques $K^u$ and $K^v$.
    By Lemma~\ref{Lemma: Clique-Sub Neighbourhood} each vertex in $H$ is incident to exactly $1$ red edge.

    Let $P$ be a longest induced path in $H$.
    Since $P$ is an induced path, $P$ is triangle-free.
    Thus, for any clique $K^u$, at most $2$ vertices of $K^u$ appear in $P$.
    Hence, each vertex in $P$ is incident to at most $1$ blue edge in $P$.
    It follows that every non-leaf vertex in $P$ is incident to at least $1$ red edge in $P$.

    Letting $r$ be the number of red edges in $P$, this implies $|P| \leq 2r+2$.
    Suppose $e_1,\dots, e_r$ are the red edges of $P$ as they appear in order.
    For each $i$ let $e_i = (v_{i},v_{i+1})(v_{i+1},v_i)$.
    Then, $v_1v_2\dots v_{r+1}$ is a path in $G$.
    It follows that $r+1 \leq p$ implying that $|P| \leq 2(p-1)+2 = 2p$.
    This completes the proof.
\end{proof}

The next lemma is proven in \cite{joret2010cops}.

\begin{lemma}[Lemma~2.2 \cite{joret2010cops}]\label{Lemma: Clique-Sub cop num}
    If $H$ is the clique substitution  of $G$, then $c(G) \leq c(H)$.
\end{lemma}

We are now prepared to prove Theorem~\ref{Thm: path to claw reduction}.

\begin{proof}[Proof of Theorem~\ref{Thm: path to claw reduction}]
    Let $G$ be a graph whose longest path is length $p$ and $c(G) \geq t$.
    Let $H$ be the clique substitution of $G$.
    By Lemma~\ref{Lemma: Clique-Sub Neighbourhood} the neighbourhood of no vertex induces a $\claw$ or a $\butterfly$.
    Both of these graphs have a universal vertex so $H$ is $(\claw,\butterfly)$-free.
    By Lemma~\ref{Lemma: Clique-Sub C_4,C_5 free} $H$ is $(C_4,C_5)$-free and by Lemma~\ref{Lemma: Clique-Sub induced paths} $H$ is $P_{2p+1}$-free.
    By Lemma~\ref{Lemma: Clique-Sub cop num} $c(H) \geq t$.
    Thus, $H$ is a $(P_{2p+1},\claw,\butterfly,C_4,C_5)$-free graph with $c(G)\geq t$.
\end{proof}

\section{$(P_k,\E \,)$-free Graphs}

The goal of this section is to show how $\lceil\frac{k-1}{2} \rceil+3$ cops can capture the robber in an
$\E$-free graph.
Our argument will be reminiscent of the one in Section~3.
As a result we begin by considering which induced flails can exist in an $\E$-free graphs.

\begin{figure}[h!]
    \centering
    \includegraphics[scale = 1.05]{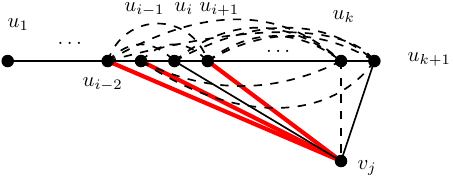}
    \caption{
    The figure depicts the flail discussed in Lemma~\ref{Lemma: E-free flail}. The lemmas claim at least one of a certain set of edges must exist,
    these edges are drawn as red and bold.
    }
    \label{fig:EFlail}
\end{figure}

\begin{lemma}\label{Lemma: E-free flail}
    Let $G$ be a $\E$-free graph and $H$ an induced subgraph of $G$.
    If $H$ is a $(k,t,S)$-flail
    such that $k\geq 6$, and
    \[\{ (k,j): j \in [t]\} \cap S = \emptyset,\]
    and $(i,j),\in S$ for some $3\leq i \leq k-3$ and $j$, 
    then $\{(i-1,j),(i+1,j)\} \cap S \neq \emptyset$ or $(i-2,j)\in S$.
\end{lemma}

\begin{proof}
     Let $G$ be a $\E$-free graph and $H$ an induced subgraph of $G$.
     For contradiction suppose that 
     $H$ is a $(k,t,S)$-flail
    such that $k\geq 6$, and
    \[\{ (k,j): j \in [t]\} \cap S = \emptyset,\]
    and $(i,j),\in S$ for some $3\leq i \leq k-3$ and $j$.
    Moreover, suppose that 
    \[
    (i-2,j),(i-1,j),(i+1,j) \notin S.
    \]
    Then $H[\{u_{i-2},u_{i-1},u_i,u_{i+1},v_j,u_{k+1},u_k\}]$ induces a graph isomorphic to $\E$.
    This contradicts $G$ being $\E$-free.
\end{proof}

We are now prepared to prove Theorem~\ref{Thm: E free}.
The proof is very similar to that of Theorem~\ref{Thm: path,claw,butterfly free}, 
however there are more cops, and the cops' strategy for capturing is different.

\begin{proof}[Proof of Theorem~\ref{Thm: E free}]

Let $G$ be a $(P_k,\E\,)$-free graph.
We will show how  $\lceil\frac{k-1}{2}\rceil+3$ cops can capture the robber, no matter how the robber plays.
The cops begin the game with all $\lceil\frac{k-1}{2}\rceil+3$ cops on a fixed but arbitrary vertex $w_0$.
Let $v_0$ denote the starting position of the robbers.
Label the cops $C^\uparrow,C^\downarrow, C^\Downarrow, C^0,\dots, C^{\lceil\frac{k-1}{2}\rceil-1}$.
We denote the robber by $R$.

If $\dist(w_0,v_0) = 1$, then the cops capture on their first turn.
Assume without loss of generality that the robber never deliberately moves adjacent to a cop, as this is losing for the robber.
Then $\dist(w_0,v_0)\geq 2$.
If $\dist(w_0,v_0) = 2$, then proceed to Step~2 of the cops' strategy.
Otherwise, $\dist(w_0,v_0) > 2$ in which case proceed to Step~1 of the cops' strategy.

\vspace{0.25cm}
\noindent\underline{Step.1:} We suppose $\dist(w_0,v_0) > 2$, all cops are on $w_0$, the robber is on $v_0$, and it is the cops turn to move.
\vspace{0.25cm}

This step proceeds exactly as in the proof of Theorem~\ref{Thm: path,claw,butterfly free}, except that we place a cop on every other vertex rather than on every third vertex.
That is, cop $C^j$ will occupy vertex $w_{i-2j}$ before the cops move on turn $i$.
Also, when reaching Step 2, we assume the robber is adjacent to a vertex of the path $w_0\dots w_t$, which is an easy corollary of the argument,
but was not expressly claimed in Theorem~\ref{Thm: path,claw,butterfly free}.

\vspace{0.25cm}
\noindent\underline{Step.2:} It is the cops turn, vertices $w_0\dots w_t$ form an induced path, each cop $C^j$ occupies vertex $w_{t-2j}$, 
where $w_{-x} = w_0$ for any positive $x$, 
cop $C^\uparrow$ occupies vertex $w_t$,
cops $C^\downarrow$ and $C^\Downarrow$ occupy vertex $w_0$,
and the robber is adjacent to some vertex $w_q$.
\vspace{0.25cm}

This step proceeds similarly to Step~2 in the proof of Theorem~\ref{Thm: path,claw,butterfly free}, but with subtle differences.
As a result we focus primarily on the differences between the cops' strategy here versus Theorem~\ref{Thm: path,claw,butterfly free}.

If $w_q$ contains a cop, then the game is over.
Suppose then that $w_q$ does not contain a cop.
Hence, there is a cop on $w_{q-1}$ and $w_{q+1}$.

As in Theorem~\ref{Thm: path,claw,butterfly free}, we designate a `lead cop' $\mathcal{C}^0$. 
Say this is the cop on $w_{q+1}$ without loss of generality.
Let $\mathcal{C}^\uparrow$ and $\mathcal{C}^\downarrow$ be defined as in Theorem~\ref{Thm: path,claw,butterfly free}.
We let $u_i$ and $r_i$ be defined as in Theorem~\ref{Thm: path,claw,butterfly free}.

All cops walk to $w_q$ along the path $w_0\dots w_t$.
Notice that before the cops move in turn $2$ there is already a cop,
distinct from $\mathcal{C}^0$,$\mathcal{C}^\uparrow$ or $\mathcal{C}^\downarrow$, on the vertex $u_1$.
This cop began on $w_{q-1}$.
Let $\mathcal{C^*}$ be the label for this cop.
Before moving on turn $2$, $\mathcal{C}^0$ is on vertex $w_q = u_1$.
On turn $2$, $\mathcal{C}^0$ moves to $r_0$ setting $r_0 = u_2$.
Cop $\mathcal{C^*}$ remains on $u_1$.

In future turns the cop $\mathcal{C}^0$ will move as in the proof of Theorem~\ref{Thm: path,claw,butterfly free},
building a Gy\'arf\'as path $u_1\dots u_i$ by chasing the robber.
Notice that in Theorem~\ref{Thm: path,claw,butterfly free}, $\mathcal{C}^0$ builds a Gy\'arf\'as path $u_0\dots u_i$,
so the game may proceed one extra turn, since it takes the cop $\mathcal{C}^0$ in this proof one turn to get into position to begin their Gy\'arf\'as path.
The cop $\mathcal{C}^*$ always remains one step behind $\mathcal{C}^0$, 
so before the cops move in turn $i+1$, $\mathcal{C}^*$ occupies vertex $u_{i-1}$.

Since the remaining cops proceed to $u_1$ along $w_0\dots, w_t$,
two cops will arrive on each odd number turn,
until all the cops from $w_0\dots w_{q-1}$ or $w_{q+1}\dots w_{t}$
have arrived.
On the final odd numbered turn where cops are arriving along $w_0\dots w_{q-1}$ and $w_{q+1}\dots w_{t}$
three cops will arrive.
This is because either $\mathcal{C}^\uparrow$ or $\mathcal{C}^\downarrow$
will also arrive on this turn.
The first of $\mathcal{C}^\uparrow$ or $\mathcal{C}^\downarrow$ to arrive behaves like all other cops,
while the second to arrive remains on $u_1$ for the rest of the game, unless to capture.

Once cops arrive at $u_1$,
they move in single file 
along the path $u_1\dots u_i$, so that each vertex $u_{i-2j}$ contains a cop,
with one caveat: 
if $i\geq 2$ is even, then when moving in turn $i+1$ the cop which is next in line moves to $u_2$ one turn early, so that there is always a cop on $u_2$. 
Call this cop $\mathcal{C}$.
On the subsequent turn $i+2$, $\mathcal{C}$ remaining on $u_2$.
After this, $\mathcal{C}$ continues along the path $u_1\dots u_i$ moving to a new vertex each turn.
Since two cops arrive at $u_1$ on each odd numbered turn until an odd numbered turn where three cops arrive 
it is trivial to verify this strategy implies that
for all $i \leq k-2$  prior to moving in turn $i+1$ every vertex
\[
\{u_1,u_2\} \cup \{u_{i-2j}: j\geq 0\} \cup \{u_{i-1}\}
\]
contains a cop.
Our assumption that $i\leq k-2$ is key since if the path is much longer than this, we will not have enough cops to cover the entire path.

Suppose $1 \leq i \leq k-2$ is the smallest integer such that, while the cops are following this strategy,
the robber's vertex $r_i$ has a neighbour
on the path $u_1\dots, u_{i}$.
Observe that as until this happens the cops can continue growing their Gy\'arf\'as path by chasing the robber along the robber's previously visited vertices,
so it is safe to assume $u_1 \dots u_i$ is an induced path.
When $i \leq 5$, there is a cop on every vertex of $\{u_1,\dots, u_i\}$ since 
\[
\{u_1,\dots, u_i\}\subseteq \{u_1,u_2\} \cup \{u_{i-2j}: j\geq 0\} \cup \{u_{i-1}\}.
\]
Hence,  $i\geq 6$ as otherwise the robber is adjacent to a cop on the cops turn.

Without loss of generality suppose $\ell$ is the least integer such that $u_\ell \in N(r_i)$.
Suppose without loss of generality that the robber is not adjacent to a cop.
Then that $u_{\ell}$ does not contain a cop 
implying $3 \leq \ell \leq i-3$ and $\ell \not\equiv i \pmod{2}$.
Since $i$ is chosen to be as small as possible
$u_1\dots u_{i}r_{i-1}$ is an induced path,
and 
$H = G[\{u_1,\dots, u_{i},r_{i-1}\} \cup \{r_i\}]$ is a $(i,1,S)$-flail
where 
\[
S \cap \left(\{1,2\} \cup \{ i -2j: j\geq 0\} \cup \{i-1\} \right) = \emptyset.
\]
Hence, $\{(\ell-1,1),(\ell+1,1)\}\cap S = \emptyset$.
Since $G$ is $\E$-free, $H$ is $\E$-free.
So by Lemma~\ref{Lemma: E-free flail}, $(\ell-2,1)\in S$, implying $u_{\ell-2}\in N(r_i)$.
This contradicts the minimality of $\ell$.

Therefore, 
we conclude that if $i\leq k-2$,
the robber moves to a vertex $r_i$ such that $N(r_i) \cap \{u_1,\dots, u_i\} \neq \emptyset$,
then they are moving adjacent to a cop.
Since this is losing for the robber, the robber will delay such a move as long as possible.
But each time the robber delays moving adjacent to $\{u_1,\dots, u_i\}$
this allows the cops to make a longer induced path by tracing the steps of the robber.
If this lasts until $i=k-1$,
then $u_1\dots u_{k-2}r_{k-3} r_{k-2}$ is an induced path of length $k$ in $G$,
contradicting that $G$ is $P_k$-free.
This concludes the proof.
\end{proof}

\section{Future Work}

We conclude with a discussion of open problems.
Given, Theorem~\ref{Thm: main longest path} it is natural to ask if there are graphs whose longest path is length $p$ with cop number $\Omega(p)$.
Recall that proving no such graphs exist, that is $c(G) = O(p^{1-\epsilon})$ for all graphs $G$ whose longest path is length $p$,
would imply the weak Meyniel conjecture.

\begin{conjecture}\label{Conj: linear lower in path length}
    There exists an $\epsilon>0$ such that 
    for all integer $p\geq 1$, there is a graph $G$ whose longest path is length $p$ with $c(G)\geq \epsilon p$.
\end{conjecture}

Next we recall that
Theorem~\ref{Thm: E free} implies that all $(P_k,\claw)$-free graphs have cop number at most $\lceil \frac{k-1}{2}\rceil+3$.
Do there exists $(P_k,\claw)$-free graphs with cop number $(\frac{1}{2}-o(1))k$?
We note that the random, diameter $2$, $P_k$-free graphs with cop number $\lfloor\frac{k-1}{2}\rfloor$
constructed in \cite{clow2025cops} are not $\claw$-free with high probability.
It seems hard to construct $P_k$-free graphs with large cop number, 
so deciding if there are $(P_k,\claw)$-free graphs with such a cop number may be out of immediate reach.
Instead we make an easier to prove conjecture.

\begin{conjecture}\label{Conj: P_k,claw}
    There exists an $\epsilon>0$ such that 
    for all integer $k\geq 1$, there is a $(P_k,\claw)$-free graph $G$ with $c(G)\geq \epsilon k$.
\end{conjecture}

Of course, it would also be of interest if one can prove there are $P_k$-free graphs with cop number more than 
$\lfloor\frac{k-1}{2}\rfloor$ when $k\geq 6$.
Do such graphs exist?
If not, then demonstrating this would prove a much stronger, and best possible, version of Sivaraman's conjecture.

\begin{problem}\label{Prob: P_k}
    For all $k\geq 6$ demonstrate a $P_k$-free graph whose cop number is greater than $\lfloor\frac{k-1}{2}\rfloor$, or prove no such graphs exist.
\end{problem}

\bibliographystyle{abbrv}
\bibliography{bib}

\end{document}